\documentclass[12pt,leqno,a4paper,final]{amsart}
\usepackage{amsmath, verbatim, color}
\usepackage{mathrsfs}
\usepackage{dsfont}
\usepackage[a4paper,margin=1in]{geometry}

\usepackage[notcite,notref]{showkeys}

\theoremstyle{plain}
\newtheorem{theorem}{Theorem}[section]

\newtheorem{lemma}[theorem]{Lemma}

\theoremstyle{definition}
\newtheorem{remark}[theorem]{Remark}

\newtheorem*{ack}{Acknowledgements}
\numberwithin{equation}{section}

\newcommand\E{\mathds E}

\newcommand\N{\mathds N}
\newcommand\Z{\mathds Z}
\newcommand\R{\mathds R}

\renewcommand\P{\mathds P}

\renewcommand\d{\mathrm{d}}
\newcommand\e{\mathrm{e}}

\begin{document}\allowdisplaybreaks
\title[Subgeometric Rates of Convergence under
Discrete Time Subordination]{\bfseries Subgeometric Rates
of Convergence for Discrete Time Markov Chains under
Discrete Time Subordination}

\author[C.-S.~Deng]{Chang-Song Deng}
\address[C.-S.~Deng]{School of Mathematics and Statistics\\ Wuhan University\\ Wuhan 430072, China}
\email{dengcs@whu.edu.cn}
\thanks{Financial support through the
National Natural Science Foundation
of China (11401442, 11831015) is gratefully acknowledged.}


\maketitle

\begin{abstract}
    In this paper, we are concerned with the subgeometric
    rate of convergence of a Markov chain
    with discrete time parameter to its invariant
    measure in the $f$-norm. We clarify how three typical
    subgeometric rates of convergence are inherited
    under a discrete time version of
    Bochner's subordination. The crucial point
    is to establish the corresponding moment estimates
    for discrete time subordinators under some reasonable
    conditions on the underlying Bernstein function.
\end{abstract}

\section{Introduction}

This article is a continuation of the very recent
work \cite{DSS17}, where
subgeometric rates of convergence are established
for continuous time Markov processes under subordination
in the sense of Bochner, and it aims to derive the analogous
result when the time parameter is discrete. Readers
are urged to refer to \cite[Chapter 5]{Chen05} and
\cite[Chapters 13 and 15]{MT93} for some background
on the topic of convergence rates of Markov processes.
For recent developments on
subgeometric ergodicity,
see e.g.\ \cite{But14, dou09, DFMS04, DMS07,
for05, RW01}.

First, we recall the notion of discrete
time subordinator, which is a discrete time counterpart
of the classical continuous time
subordinator (i.e. nondecreasing L\'{e}vy process on $[0,\infty)$) and was initialed
in \cite{BS12}; see also
\cite{BC15, BC15b, BCT17, Mi16, MS17} for
further developments on random walks under
discrete time subordination.
A function $\phi:(0,\infty)\rightarrow[0,\infty)$
is a Bernstein function if $\phi$ is
a $C^\infty$-function satisfying
$(-1)^{n-1}\phi^{(n)}\geq0$
for all $n\in\N$ (here $\phi^{(n)}$ denotes
the $n$-th derivative of $\phi$). It is well known,
see e.g.\ \cite[Theorem 3.2]{SSV12}, that every
Bernstein function has a unique L\'{e}vy--Khintchine
representation
$$
    \phi(x)
    =a+bx+\int_{(0,\infty)}\left(1-\e^{-xy}
    \right)\,\nu(\d y),
    \quad x>0,
$$
where $a\geq0$ is the killing term, $b\geq0$ is
the drift term and $\nu$ is a Radon measure
on $(0,\infty)$ such that
$\int_{(0,\infty)}(1\wedge y)\,\nu(\d y)<\infty$.
As usual, we make the convention that
$\phi(0):=\phi(0+)=a$. For our purpose, we will assume that
$\phi$ has no killing term (i.e.\ $a=0$); that is, $\phi$
is of the form
\begin{equation}\label{bern-11}
    \phi(x)=bx+\int_{(0,\infty)}\left(1-\e^{-xy}
    \right)\,\nu(\d y),\quad x\geq0,
\end{equation}
where $b$ and $\nu$ are as above.
Without loss of generality, we also assume
that $\phi(1)=1$; otherwise, we replace
$\phi$ by $\phi/\phi(1)$.

For $m\in\N$, we set
\begin{gather*}
    c(\phi,m)=\frac{1}{m!}\int_{(0,\infty)}
    y^m\e^{-y}\,\nu(\d y)+
    \begin{cases}
        b,&\text{if\ \ }m=1,\\
        0,&\text{if\ \ }m\geq2.
    \end{cases}
\end{gather*}
Since
$$
    \sum_{m=1}^\infty c(\phi,m)=b+\sum_{m=1}^\infty
    \frac{1}{m!}\int_{(0,\infty)}
    y^m\e^{-y}\,\nu(\d y)
    =b+\int_{(0,\infty)}\left(1-\e^{-y}\right)\,\nu(\d y)
    =\phi(1)=1,
$$
we know that $\{c(\phi,m):m\in\N\}$ gives rise to
a probability measure on $\N$, and hence we can define
a random walk $T=\{T_n:n\in\N\}$ on $\N$ by
$T_n:=\sum_{k=1}^nR_k$, where $\{R_k:k\in\N\}$ is a sequence
of independent and identically distributed random variables
with $\P(R_k=m)=c(\phi,m)$ for $k,m\in\N$. As a strictly
increasing process, the random
walk $T$ is called a discrete time
subordinator associated with the Bernstein
function $\phi$. Obviously, $T_n\geq n$, and
for $m,n\in\N$ with $m\geq n$,
$$
    \P(T_n=m)=\sum_{m_1+\dots+m_n=m}
    \prod_{i=1}^nc(m_i,\phi).
$$

Let $X=\{X_n:n\in\N\}$ be a Markov chain on a general
measurable state
space $E$, and denote by $P^n(x,\d y)$ the
$n$-step transition kernel. Throughout this paper, we
always assume that $X$ and $T$ are independent, and
that $X$ has an invariant measure $\Pi$:
$$
    \int_EP^n(x,\cdot)\,\Pi(\d x)=\Pi
    \quad \text{for all $n\in\N$.}
$$
The subordinate process is given by the random time-change
$X_n^\phi:=X_{T_n}$. The process $X^\phi=\{X_n^\phi:n\in\N\}$
is again a Markov chain, and it follows easily
from the independence of $X$ and $T$ that the
$n$-step transition kernel of $X^\phi$ is
\begin{equation}\label{tran65}
    P^n_\phi(x,\d y)=\sum_{m=n}^\infty P^m(x,\d y)
    \P(T_n=m).
\end{equation}
This implies that $\Pi$ is also invariant for the
time-changed chain $X^\phi$.

For a (measurable) control function
$f:E\rightarrow[1,\infty)$, the
$f$-norm (cf. \cite[Chapter 14]{MT93}) of a signed
measure $\mu$ on $E$ is defined as
$\|\mu\|_f:=\sup_{|g|\leq f}|\mu(g)|$, where
the supremum ranges over all measurable functions
$g:E\rightarrow\R$ with $|g|\leq f$, and
$\mu(g)=\int_Eg\,\d\mu$. If $f\equiv1$,
then the $f$-norm $\|\cdot\|_f$ reduces
to the total variation
norm $\|\cdot\|_{\operatorname{TV}}$;
since $f\geq1$, we always have $\|\cdot\|_f\geq
\|\cdot\|_{\operatorname{TV}}$; if furthermore $f$ is bounded
then these two norms are equivalent.

It is said that the process $X$ has subgeometric convergence
in the $f$-norm if
\begin{equation}\label{rate1}
    \left\|P^n(x,\cdot)-\Pi\right\|_f\leq C(x)r(n),
    \quad x\in E,\,n\in\N,
\end{equation}
where $C(x)$ is a positive constant depending
on $x\in E$ and
$r:\N\rightarrow(0,1]$ is a
nonincreasing function with $r(n)\downarrow0$
and $\log r(n)/n \uparrow 0$ as $n\rightarrow\infty$.
Here, $r$ is called the subgeometric rate.
In many specific models, the convergence
rate $r$ can be explicitly given
and typical examples contain
\begin{equation}\label{typical}
    r(n)=\e^{-\theta n^\delta},\quad
    r(n)=n^{-\beta},\quad
    r(n)=\log^{-\gamma}(2+n),
\end{equation}
where $\theta>0$, $\delta\in(0,1)$ and $\beta,\gamma>0$ are
some constants.
Let $\Z_+=\N\cup\{0\}$ and $\{p_k:k\in\Z_+\}$ be a
sequence such that $p_0=1$, $p_k\in(0,1)$ for all
$k\in\N$, and
$\lim_{k\rightarrow\infty}\prod_{i=1}^kp_i=0$. Consider
the backward recurrence time
chain (cf.\ \cite[Section 3.3.1]{MT93})
on the countable state space $\Z_+$ with one-step
transition kernel $P$ given by $P(k,k+1)=1-P(k,0)=p_k$
for all $k\in\Z_+$. Then this chain admits the
convergence rates in \eqref{typical}
under some assumptions,
see \cite[Section 3.1]{DFMS04} for details.

In recent years, there has been an increasing
interest in the stability of properties
of continuous time Markov processes and their
semigroups under Bochner's subordination.
See \cite{GRW11} for the
dimension-free Harnack inequality for subordinate
semigroups, \cite{DS15b} for shift Harnack inequality
for subordinate semigroups, \cite{DS15a} for
the quasi-invariance property of Brownian motion
under random time-change, and \cite{DSS17}
for subgeometric rates of convergence for continuous
time Markov processes under continuous
time subordination. Subordinate functional
inequalities can be found
in  \cite{BM07, GM15, SW12}.

It is a natural question whether subgeometric rates
of convergence can be preserved under discrete
time subordination. If $P^n$ is subgeometrically
convergent to $\Pi$ in the $f$-norm, is it possible
to derive quantitative
bounds on the convergence rates of the subordinate
Markov chain $X^\phi$? What we are going to do is to
find some function
$r_\phi:\N\rightarrow(0,1]$ such
that $\lim_{n\rightarrow\infty}r_\phi(n)=0$ and
\begin{equation}\label{rate2}
    \left\|P^n_\phi(x,\cdot)-\Pi\right\|_f\leq C(x)r_\phi(n),
    \quad x\in E,\,n\in\N
\end{equation}
for some constant $C(x)>0$ depending only
on $x\in E$. As in \cite{DSS17}, it turns out that
if the convergence rates of the original
chain $X$ are of the three typical forms
in \eqref{typical},
then we are able
to obtain convergence rates for the
subordinate Markov chain $X^\phi$ under some reasonable
assumptions on the underlying Bernstein function.

The main result of this note is the following.
As usual, denote by $\phi^{-1}$ the inverse function
of the (strictly increasing) Bernstein
function $\phi$.

\begin{theorem}\label{main}
Let $X$ be a discrete time Markov chain and
$T$ an independent discrete time subordinator associated
with Bernstein function $\phi$ given
by \eqref{bern-11} such that $\phi(1)=1$.

\begin{enumerate}
\item[\bfseries\upshape a)]
    Assume that \eqref{rate1} holds with rate $r(n)=\e^{-\theta n^\delta}$ for some constants $\theta>0$ and $\delta\in(0,1]$.
    If
    \begin{equation}\label{ahg54}
        \nu(\d y)\geq cy^{-1-\alpha}\,\d y
    \end{equation}
    for some constants $c>0$ and $\alpha\in(0,1)$, then \eqref{rate2} holds with rate
    $$
        r_\phi(n)=\exp\left[
            -C\,n^{
            \frac{\delta}{\alpha(1-\delta)+\delta}
            }
        \right],
    $$
    where $C=C(\theta,\delta,c,\alpha)>0$.

\item[\bfseries\upshape b)]
    Assume that \eqref{rate1} holds with rate $r(n)=n^{-\beta}$ for some constant $\beta>0$. If
    \begin{equation}\label{bern}
        \liminf_{x\rightarrow\infty}\frac{\phi(x)}{\log x}>0
        \quad\text{and}\quad
        \limsup_{x\downarrow0}\frac{\phi(\lambda x)}{\phi(x)}>1
        \quad\text{for some $\lambda>1$},
    \end{equation}
    then \eqref{rate2} holds with rate
    $$
        r_\phi(n)=\left[
        \phi^{-1}\left(\frac 1n\right)
        \right]^\beta.
    $$

\item[\bfseries\upshape c)]
    Assume that \eqref{rate1} holds with rate $r(n)=\log^{-\gamma}(2+n)$ for some
    constant $\gamma>0$. Then \eqref{rate2} holds
    with rate
    $$
        r_\phi(n)=\log^{-\gamma}(2+n).
    $$
\end{enumerate}
\end{theorem}

\begin{remark}
    \textbf{\upshape a)}
            According to \cite[Lemma 2.2\,(ii)]{DSS17},
            the second
            condition in \eqref{bern} is equivalent to
            $$
                \limsup_{x\downarrow0}\frac{\phi(\lambda x)}{\phi(x)}>1
                \quad\text{for all $\lambda>1$}.
            $$

        \medskip\noindent\textbf{\upshape b)}
            Let $b=0$ and $\nu(\d y)=\frac{\alpha}
            {\Gamma(1-\alpha)}\,y^{-1-\alpha}\,\d y$
            with $\alpha\in(0,1)$. In this case, it
            is clear that \eqref{ahg54} holds,
            and by the
            formula \cite[p.\ vii]{SSV12}
            $$
                x^\alpha=\frac{\alpha}{\Gamma(1-\alpha)}
                \int_0^\infty\left(1-\e^{-xy}\right)
                y^{-1-\alpha}\,\d y,\quad x>0,\alpha\in(0,1),
            $$
            we know that the corresponding
            Bernstein function \eqref{bern-11} is given by
            the fractional power function $\phi(x)=x^\alpha$.
            One can construct more examples for
            \eqref{ahg54} by choosing
            $\nu(\d y)=cy^{-1-\alpha}\,\d y
            +\tilde{\nu}(\d y)$, where $c>0$,
            $\alpha\in(0,1)$, and $\tilde{\nu}$
            is another L\'{e}vy measure on $(0,\infty)$
            such that
            $$
                \phi(1)=b+c\alpha^{-1}\Gamma(1-\alpha)
                +\int_{(0,\infty)}
                \left(1-\e^{-y}\right)\,
                \tilde{\nu}(\d y)=1.
            $$

        \medskip\noindent\textbf{\upshape c)}
            As pointed out in \cite[Remark 1.1]{DSS17},
            typical examples for Bernstein function $\phi$
            satisfying \eqref{bern} are
            \begin{itemize}
                \item
                   $\phi(x)=\log(1+x)/\log2$;

                \item
                   $\phi(x)=x^\alpha\log^\beta(1+x)/\log^\beta2$ with $\alpha\in(0,1)$
                   and $\beta\in[0,1-\alpha)$;

                \item
                   $\phi(x)=x^\alpha\log^{-\beta}(1+x)/\log^{-\beta}2$ with
                   $0<\beta<\alpha<1$;

                \item
                    $\phi(x)=2^\alpha x(1+x)^{-\alpha}$ with $\alpha\in(0,1)$.
          \end{itemize}
    See \cite[Chapter 16]{SSV12} for more examples of such
    Bernstein functions.
\end{remark}

The rest of this paper is organized as follows. Section 2
is devoted to three types of moment estimates
for discrete time subordinators, which will be crucial
for the proof of Theorem \ref{main}; we stress that
this part is of some interest on its own. In Section 3,
we present the proof of Theorem \ref{main}. Finally,
we give in the Appendix an elementary inequality,
which has been used in Section \ref{mom}.

\section{Moment estimates for
discrete time subordinators}\label{mom}

Recall that a continuous time
subordinator $S=\{S_t:t\geq0\}$
associated with Bernstein function $\phi$
is a nondecreasing L\'{e}vy process taking values in $[0,\infty)$ and with Laplace transform
$$
    \E\,\e^{-uS_t}=\e^{-t\phi(u)},\quad u,t\geq0.
$$
The following result concerning moment estimates
for continuous time subordinators is
taken from \cite[Theorem 2.1]{DSS17}.

\begin{lemma}\label{moments123}
Let $S$ be a continuous time subordinator
associated with Bernstein function $\phi$
given by \eqref{bern-11}.
\begin{enumerate}
\item[\bfseries\upshape a)]
    Let $\theta>0$ and $\delta\in(0,1]$.
    If \eqref{ahg54} holds for some constants $c>0$ and $\alpha\in(0,1)$, then there exists a constant $C=C(\theta,\delta,c,\alpha)>0$ such that
    $$
        \E\,\e^{-\theta S_t^\delta}
        \leq\exp\left[
            -C\,t^{
            \frac{\delta}{\alpha(1-\delta)+\delta}
            }
        \right]
        \quad\text{for all sufficiently large $t>1$}.
    $$

\item[\bfseries\upshape b)]
        Let $\beta>0$. If the Bernstein function $\phi$ satisfies \eqref{bern}, then there exists a constant $C=C(\beta)>0$ such that
            $$
                \E S_t^{-\beta}\leq C\left[
                \phi^{-1}\left(\frac1t\right)\right]^\beta
                \quad \text{for all sufficiently
                large $t>1$}.
            $$
        \end{enumerate}
\end{lemma}

Analogous to Lemma \ref{moments123}, we shall
establish the corresponding results for discrete
time subordinators. For related moment estimates
for general L\'{e}vy(-type) processes,
we refer to \cite{DS15b, Ku17}.

Our main contribution
in this section is the following result.

\begin{theorem}\label{moments}
Let $T$ be a discrete time subordinator
associated with Bernstein function $\phi$
given by \eqref{bern-11}
such that $\phi(1)=1$.
\begin{enumerate}
\item[\bfseries\upshape a)]
    Let $\theta>0$ and $\delta\in(0,1]$.
    If \eqref{ahg54} holds  for some constants $c>0$ and $\alpha\in(0,1)$, then there exists a constant $C=C(\theta,\delta,c,\alpha)>0$ such that
    $$
        \E\,\e^{-\theta T_n^\delta}
        \leq\exp\left[
            -C\,n^{
            \frac{\delta}{\alpha(1-\delta)+\delta}
            }
        \right]
        \quad\text{for all sufficiently large $n\in\N$}.
    $$

\item[\bfseries\upshape b)]
        Let $\beta>0$. If the Bernstein function $\phi$ satisfies \eqref{bern}, then there exists a constant $C=C(\beta)>0$ such that
            $$
                \E T_n^{-\beta}\leq C\left[
                \phi^{-1}\left(\frac1n\right)\right]^\beta
                \quad \text{for all $n\in\N$}.
            $$
        \end{enumerate}
\end{theorem}

In order to prove Theorem \ref{moments}, we first
present a general result to
bound the completely monotone moment of a discrete time subordinator by that
of a continuous time subordinator.

A function $g:(0,\infty)\rightarrow\R$ is called
a completely monotone function if $g$ is of class
$C^\infty$ and $(-1)^ng^{(n)}\geq0$ for
all $n=0,1,2,\dots$, see \cite[Chapter 1]{SSV12}.
By the celebrated theorem of
Bernstein (cf.\ \cite[Theorem 1.4]{SSV12}),
every completely monotone function is
the Laplace transform of a unique measure on $[0,\infty)$.
More precisely, if $g$ is a completely monotone function, then
there exists a unique measure $\mu$ on $[0,\infty)$
such that
\begin{equation}\label{repre}
    g(x)=\int_{[0,\infty)}\e^{-xt}\,\mu(\d t) \quad
    \text{for all $x>0$}.
\end{equation}

Since the function $x\mapsto x^\delta$ ($\delta\in(0,1]$)
is a (complete) Bernstein function,
it follows easily from \cite[Theorem 3.7]{SSV12} that
the following functions
\begin{equation}\label{functions}
    x\mapsto\e^{-\theta x^\delta},\quad
    x\mapsto x^{-\beta}
\end{equation}
are completely monotone functions, where $\theta>0$,
$\delta\in(0,1]$, and $\beta>0$. Indeed,
one has for $\theta>0$ and $\delta\in(0,1)$
(see \cite{Pol46}),
$$
    \e^{-\theta x^\delta}=\int_0^\infty
    \e^{-xt}\psi(\theta,\delta,t)\,\d t, \quad
    x>0,
$$
where
$$
    \psi(\theta,\delta,t)=\pi^{-1}\theta^{-1/\delta}
    \int_0^\infty\e^{-\theta^{-1/\delta}tu}
    \e^{-u^\delta\cos\pi\delta}
    \sin\left(
    u^\delta\sin\pi\delta
    \right)\,\d u;
$$
moreover, for $\beta>0$,
$$
    x^{-\beta}=
    \frac{1}{\Gamma(\beta)}\int_0^\infty
    \e^{-xt}t^{\beta-1}\,\d t,\quad x>0.
$$

\begin{lemma}\label{cihf3f}
    Let $T$ be a discrete time subordinator
    with Bernstein function $\phi$ given by \eqref{bern-11}
    such that $\phi(1)=1$. Let $S$ be a continuous time
    subordinator with the same Bernstein function $\phi$.
    If $g:(0,\infty)\rightarrow\R$ is a completely
    monotone function, then
    $$
        \E\,g(T_n)\leq\E\,g(S_n)\quad \text{for all $n\in\N$}.
    $$
\end{lemma}

\begin{proof}
    By the representation formula \eqref{repre} and
    Tonelli's theorem,
    \begin{align*}
            \E\,g(T_n)&=\E\left[
            \int_{[0,\infty)}\e^{-T_nt}\,\mu(\d t)
            \right]\\
            &=\int_{[0,\infty)}\E\,\e^{-T_nt}\,\mu(\d t)\\
            &=\int_{[0,\infty)}\prod_{k=1}^n\E\,\e^{-tR_k}\,
            \mu(\d t)\\
            &=\int_{[0,\infty)}\left(\E\,\e^{-tR_1}\right)^n\,
            \mu(\d t).
    \end{align*}
    Note that for $t\geq0$,
    \begin{align*}
            \E\,\e^{-tR_1}&=\sum_{m=1}^\infty
            \e^{-tm}c(\phi,m)\\
            &=b\e^{-t}+\sum_{m=1}^\infty
            \e^{-tm}\frac{1}{m!}\int_{(0,\infty)}
            y^m\e^{-y}\,\nu(\d y)\\
            &=b\e^{-t}+\int_{(0,\infty)}
            \left(\e^{y\e^{-t}}
            -1\right)\e^{-y}\,\nu(\d y)\\
            &=\left[b+\int_{(0,\infty)}
            \left(1-\e^{-y}\right)\nu(\d y)
            \right]-\left[
            b\left(1-\e^{-t}\right)+\int_{(0,\infty)}
            \left(1-\e^{-(
            1-\e^{-t})y}\right)
            \nu(\d y)
            \right]\\
            &=\phi(1)-\phi\left(1-\e^{-t}\right)\\
            &=1-\phi\left(1-\e^{-t}\right),
    \end{align*}
    which does not exceed $\e^{-\phi(t)}$ according
    to Lemma \ref{elemen} in the Appendix. Then we
    have for all $n\in\N$,
    \begin{align*}
        \E\,g(T_n)&\leq\int_{[0,\infty)}
        \e^{-n\phi(t)}\,\mu(\d t)\\
        &=\int_{[0,\infty)}
        \E\,\e^{-tS_n}\,\mu(\d t)\\
        &=\E\left[
        \int_{[0,\infty)}\e^{-tS_n}\,\mu(\d t)
        \right]\\
        &=\E\,g(S_n),
    \end{align*}
    which was to be proved.
\end{proof}

\begin{proof}[Proof of Theorem \ref{moments}]
    Since the functions given in \eqref{functions}
    are completely monotone functions, we only need to
    combine Lemma \ref{cihf3f} with
    Lemma \ref{moments123} to get the
    desired estimates.
\end{proof}

Since $T_n\geq n$, the following lemma is clear.

\begin{lemma}\label{khgdd}
    Let $T$ be a discrete time subordinator
    associated with Bernstein function $\phi$
    given by \eqref{bern-11}
    such that $\phi(1)=1$. For any $\gamma>0$ and
    $n\in\N$,
    $$
        \E \log^{-\gamma}(2+T_n)\leq\log^{-\gamma}(2+n)
    $$
\end{lemma}

\section{Proof of Theorem \ref{main}}

\begin{lemma}\label{generalrate}
    If \eqref{rate1} holds with some rate $r(n)$, then
    so does \eqref{rate2} with
    rate $r_\phi(n)=\E\,r(T_n)$.
\end{lemma}

\begin{proof}
    It holds from \eqref{tran65} and \eqref{rate1} that
    \begin{align*}
        \left\|P^n_\phi(x,\cdot)-\Pi\right\|_f
        &=\left\|\sum_{m=n}^\infty\left[
        P^m(x,\cdot)-\Pi
        \right]\P(T_n=m)\right\|_f\\
        &\leq\sum_{m=n}^\infty\left\|
        P^m(x,\cdot)-\Pi\right\|_f
        \P(T_n=m)\\
        &\leq C(x)\sum_{m=n}^\infty
        r(m)\P(T_n=m)\\
        &=C(x)\E\,r(T_n),
    \end{align*}
    and hence the claim follows.
\end{proof}

\begin{proof}[Proof of Theorem \ref{main}]
    The assertion follows immediately by
    combining Lemma \ref{generalrate} with the moment
    estimates for discrete time subordinators derived
    in Theorem \ref{moments} and Lemma \ref{khgdd}.
\end{proof}

\section{Appendix}\label{app}

If $\phi:[0,\infty)\rightarrow[0,\infty)$ is a concave
function, then it is easy to see that
\begin{equation}\label{subadd}
    \phi(tx)\geq t\phi(x)\quad \text{for all
    $t\in[0,1]$ and $x\geq0$}.
\end{equation}

\begin{lemma}\label{elemen}
    Let $\phi:[0,\infty)\rightarrow[0,\infty)$ be
    a concave function such that $\phi(1)=1$
    and $\phi$ is differentiable on $(0,1)$
    with $\phi'|_{(0,1)}\geq0$. Then
    $$
        \e^{-\phi(x)}+\phi\left(1-\e^{-x}\right)
        \geq1\quad \text{for all $x\geq0$}.
    $$
    In particular, the above inequality holds
    if $\phi$ is a Bernstein function {\upshape(}not
    necessarily with $\phi(0)=0${\upshape)}.
\end{lemma}

\begin{proof}
    Let
    $$
        \Phi(x):=\e^{-\phi(x)}+\phi\left(1-\e^{-x}\right),\quad
        x\geq0.
    $$
    It follows from \eqref{subadd} that for $x\geq1$,
    $$
        1=\phi(1)=\phi\left(\frac1x\cdot x\right)
        \geq\frac1x\phi(x),
    $$
    whence
    \begin{equation}\label{hfd54s}
        \e^{-\phi(x)}\geq\e^{-x},\quad x\geq1.
    \end{equation}
    Now we obtain from \eqref{subadd} and \eqref{hfd54s}
    that for $x\geq1$,
    \begin{align*}
        \Phi(x)&=\e^{-\phi(x)}+
        \phi\left(\left(1-\e^{-x}\right)\cdot1\right)\\
        &\geq\e^{-\phi(x)}+
        \left(1-\e^{-x}\right)\phi(1)\\
        &=1+\e^{-\phi(x)}-\e^{-x}\\
        &\geq1.
    \end{align*}

    It remains to consider the case that $x\in[0,1)$.
    Since $\phi$ is concave and differentiable on $(0,1)$,
    we know that $\phi'$ is nonincreasing on $(0,1)$.
    For $x\in(0,1)$, by the elementary inequality
    that $1-\e^{-x}<x$, we obtain
    \begin{equation}\label{gfds7v}
        \phi'\left(1-\e^{-x}\right)
        \geq \phi'(x)\geq0.
    \end{equation}
    Moreover, by \eqref{subadd} one has
    for $x\in(0,1)$,
    $$
        \phi(x)=\phi(x\cdot1)\geq x\phi(1)=x,
    $$
    which yields that
    $$
        \e^{-x}\geq\e^{-\phi(x)},\quad x\in(0,1).
    $$
    Combining this with \eqref{gfds7v},
    we find for $x\in(0,1)$,
    $$
        \Phi'(x)=\e^{-x}\phi'\left(1-\e^{-x}\right)
        -\e^{-\phi(x)}\phi'(x)\geq0.
    $$
    This implies that $\Phi$ is nondecreasing on
    $(0,1)$ and thus for all $x\in[0,1)$,
    $$
        \Phi(x)\geq \Phi(0)=
        \e^{-\phi(0)}+\phi(0)\geq1,
    $$
    which completes the proof.
\end{proof}

\begin{ack}
    The author would like to thank an anonymous referee for careful reading and useful suggestions.
\end{ack}

\end{document}